\documentclass[12pt]{amsart}

\usepackage{latexsym}
\usepackage{amsmath}
\usepackage{amssymb}
\usepackage{amscd}
\usepackage{amsthm}
\usepackage{xypic}
\xyoption{curve}
\usepackage{ifthen} 
\usepackage{hyperref} 
\usepackage{graphicx}

\theoremstyle{plain}
 \newtheorem{theorem}{Theorem}[section]
 \newtheorem{proposition}[theorem]{Proposition}
 \newtheorem{lemma}[theorem]{Lemma}
 \newtheorem{corollary}[theorem]{Corollary}
 
\theoremstyle{definition}
 \newtheorem{definition}[theorem]{Definition}

\theoremstyle{remark}
 \newtheorem{remark}[theorem]{Remark}

\DeclareMathOperator{\rank}{rank}

\makeindex

\begin{document}
\title[Rationality of moduli spaces of plane curves]{Rationality of moduli spaces of plane curves of small degree}

\author[B\"ohning]{Christian B\"ohning} 
\author[Graf v. Bothmer]{Hans-Christian Graf von Bothmer}
\author[Kr\"oker]{Jakob Kr\"oker}

\maketitle

\newcommand{\PP}{\mathbb{P}} 
\newcommand{\QQ}{\mathbb{Q}} 
\newcommand{\ZZ}{\mathbb{Z}} 
\newcommand{\CC}{\mathbb{C}} 
\newcommand{\rmprec}{\wp}
\newcommand{\rmconst}{\mathrm{const}}
\newcommand{\xycenter}[1]{\begin{center}\mbox{\xymatrix{#1}}\end{center}} 

\newboolean{xlabels} 
\newcommand{\xlabel}[1]{ 
                        \label{#1} 
                        \ifthenelse{\boolean{xlabels}} 
                                   {\marginpar[\hfill{\tiny #1}]{{\tiny #1}}} 
                                   {} 
                       } 
\setboolean{xlabels}{true} 

\begin{abstract}
We prove that the moduli space $C(d)$ of plane curves of degree $d$ (for projective equivalence) is rational except possibly if $d= 6, \: 7, \: 8, \:
11, \: 12,\:  14,\: 15,\: 16,\: 18,\: 20,\: 23,\: 24,\: 26,\: 32,\: 48$.
\end{abstract}

\section{Introduction} \xlabel{sIntro}

Let $C(d):= \PP (\mathrm{Sym}^d (\CC^3)^{\vee })/\mathrm{SL}_3 (\CC )$ be the moduli space of
plane curves of degree $d$. As a particular instance of the general question of rationality for invariant function fields under actions of connected linear
algebraic groups (see \cite{Dol0} for a survey), one can ask if $C(d)$ is always a rational space. The main results obtained in this direction in the
past can be summarized as follows:

\begin{itemize}
\item
$C(d)$ is rational for $d\equiv 0$ (mod $3$) and $d\ge 210$ (\cite{Kat89}).
\item
$C(d)$ is rational for $d\equiv 1$ (mod $3$), $d \ge 37$, and for $d\equiv 2$ (mod $3$), $d\ge 65$ (\cite{BvB08-1}).
\item
$C(d)$ is rational for $d\equiv 1$ (mod $4$) (\cite{Shep}).
\end{itemize}

Apart from these general results, rationality of $C(d)$ was known for some sporadic smaller values of $d$ for which the problem, however, can be very
hard (cf. e.g. \cite{Kat92/2}, \cite{Kat96}).\\
In this paper, using methods of computer algebra, we improve these results substantially so that only $15$ values of $d$ remain for which rationality
of $C(d)$ is open. This is the content of our main Theorem \ref{tComprehensive}. 

\

In Section \ref{sDoubleBundleAlgorithms} we discuss the algorithms used to improve the result that $C(d)$ is rational for $d\equiv 0$ (mod $3$) and $d\ge 210$ (see above) to the
degree that $C(d)$ is rational for $d\equiv 0$ (mod $3$) and $d\ge 30$ with the possible exception of $d=48$. This is the hardest part
computationally. We use the \emph{double bundle method} of \cite{Bo-Ka} and an algorithm to find matrix representatives for certain $\mathrm{SL}_3
(\CC )$-equivariant bilinear maps
\[
\psi\, :\, V \times U \to W
\]
($V$, $U$, $W$ $\mathrm{SL}_3 (\CC )$-representations) in a fast and algorithmically efficient way. It is described in Section \ref{sDoubleBundleAlgorithms}, and ultimately based
on writing a homogeneous polynomial as a sum of powers of linear forms. An immense speedup of our software was achieved by using the FFPACK-Library \cite{FFPACK} for linear algebra over finite fields.

\

In Section \ref{sCovariantAlgorithms} we describe the methods and algorithms to improve the degree bounds for $d\equiv 1$ (mod $3$) and $d\equiv 2$ (mod $3$) mentioned above:
we obtain rationality of $C(d)$ for $d\equiv 1$ (mod $3$) and $d\ge 19$ (for $d\equiv 1$ (mod $9$), $d\ge 19$, Shepherd-Barron had proven 
rationality in \cite{Shep}), and for $d\equiv 2$ (mod $3$), $d\ge 35$. This uses techniques introduced in \cite{BvB08-1} and is ultimately based on
the \emph{method of covariants} which appeared for the first time in \cite{Shep} as well as writing a homogeneous polynomial as a sum of powers of linear forms and interpolation.

\

In Section \ref{sApplications} we summarize these results, and combine them with the known results for $C(d)$ for
smaller $d$ and with the proofs of rationality for $C(10)$ and $C(27)$ (the method to prove rationality for $C(10)$ was suggested in \cite{Bo-Ka}).

\section{The Double Bundle Method: Algorithms} \xlabel{sDoubleBundleAlgorithms}

In this section we give a brief account of the so-called double bundle method, and then describe the algorithms pertaining to it that we use in our
applications. 
The main technical point is the so called "no-name lemma".

\begin{lemma}\xlabel{lNoNameLemma}
Let $G$ be a linear algebraic group with an almost free action on a variety $X$. Let $\pi \, :\, \mathcal{E} \to X$ be a $G$-vector bundle of rank
$r$ on $X$. Then one has the following commutative diagram of $G$-varieties
\xycenter{
\mathcal{E} \ar@{-->}[r]^{f} \ar[rd]_{\pi} &X\times\mathbb{A}^r \ar[d]^{\mathrm{pr}_1}\\
&X 
	}
where $G$ acts trivially on $\mathbb{A}^r$, $\mathrm{pr}_1$ is the projection onto $X$, and the rational map $f$ is \emph{birational}.
\end{lemma}

If $X$ embeds $G$-equivariantly in $\PP (V)$, $V$ a $G$-module, $G$ is reductive and $X$ contains stable points of $\PP (V)$, then this is an
immediate application of descent theory and the fact that a vector bundle in the \'{e}tale topology is a vector bundle in the Zariski topology. The
result appears in \cite{Bo-Ka}. A proof without the previous technical restrictions is given in \cite{Ch-G-R}, \S 4.3.\\
The following result (\cite{Bo-Ka}, \cite{Kat89}) is the form in which Lemma \ref{lNoNameLemma} is most often applied since it allows one to extend its
scope to irreducible representations. 

\begin{theorem}\xlabel{tDoubleBundleOriginal}
Let $G$ be a linear algebraic group, and let $U$, $V$ and $W$, $K$ be (finite-dimensional) $G$-representations. Assume that the stabilizer in general
position of $G$ in $U$, $V$ and $K$ is equal to one and the same subgroup $H$ in $G$ which is also assumed to equal the ineffectiveness kernel in these
representations (so that the action of $G/H$ on $U$, $V$, $K$ is almost free).\\
The relations $\dim U -\dim W =1$ and $\dim V - \dim U > \dim K$ are required to hold.\\
Suppose moreover that there is a $G$-equivariant bilinear map 
\[
\psi \, :\, V \times U \to W
\]
and a point $(x_0, \: y_0) \in V\times U$ with $\psi (x_0, \: y_0)=0$ and $\psi (x_0, \: U) =W$, $\psi (V, \: y_0) =W$.\\
Then if $K/G$ is rational, the same holds for $\PP (V)/G$.
\end{theorem}

\begin{proof}
We abbreviate $\Gamma := G/H$ and let $\mathrm{pr}_U$ and $\mathrm{pr}_V$ be the projections of $V\times U$ to $U$ and $V$. By the genericity
assumption on $\psi$, there is a unique irreducible component $X$ of $\psi^{-1}(0)$ passing through $(x_0, \: y_0)$, and there are non-empty open
$\Gamma$-invariant sets
$V_0\subset V$ resp.
$U_0\subset U$ where $\Gamma$ acts with trivial stabilizer and the fibres $X\cap \mathrm{pr}_V^{-1}(v)$ resp. $X\cap \mathrm{pr}_U^{-1}(u)$ have the
expected dimensions $\dim U -\dim W=1$ resp. $\dim V -\dim W$. Thus
\[
\mathrm{pr}_V^{-1}(V_0) \cap X \to V_0 , \quad \mathrm{pr}_U^{-1}(U_0) \cap X \to U_0
\]
are $\Gamma$-equivariant bundles, and by Lemma \ref{lNoNameLemma} one obtains vector bundles
\[
(\mathrm{pr}_V^{-1}(V_0) \cap X)/\Gamma \to V_0/\Gamma , \quad (\mathrm{pr}_U^{-1}(U_0) \cap X)/\Gamma \to U_0/\Gamma
\]
of rank $1$ and $\dim V -\dim W$ 
and there is still a homothetic $T:=\CC^{\ast}\times \CC^{\ast}$-action on these bundles. By a well-known theorem of Rosenlicht \cite{Ros}, the action
of the torus
$T$ on the respective base spaces of these bundles has a section over which the bundles are trivial; thus we get 
\[
\PP (V)/\Gamma \sim (\PP (U) /\Gamma ) \times \PP^{\dim V-\dim W-1} = (\PP (U) /\Gamma ) \times \PP^{\dim V-\dim U} \, .
\]
On the other hand, one may view $U\oplus K$ as a $\Gamma$-vector bundle over both $U$ and $K$; hence, again by Lemma \ref{lNoNameLemma},
\[
U/\Gamma \times \PP^{\dim K} \sim K/\Gamma \times \PP^{\dim U}\, .
\] 
Since $U/\Gamma$ is certainly stably rationally equivalent to $\PP (U)/\Gamma$ of level at most one, the inequality $\dim V -\dim U > \dim K$ insures
that $\PP (V)/\Gamma$ is rational as $K/\Gamma$ is rational. 
\end{proof}

In \cite{Kat89} this is used to prove the rationality of the moduli spaces $\PP (\mathrm{Sym}^d (\CC^3)^{\vee }) /\mathrm{SL}_3 (\CC )$ of plane curves
of degree $d \equiv 0$ (mod $3$) and $d \ge 210$. A clever inductive procedure is used there to reduce the genericity requirement for the occurring
bilinear maps $\psi$ to a purely numerical condition on the labels of highest weights of irreducible summands in $V$, $U$, $W$. This method is only
applicable if $d$ is large. We will obtain rather comprehensive results for $d \equiv 0$ (mod $3$), and $d$ smaller than $210$ by explicit computer
calculations.

In the following we put $G:=\mathrm{SL}_3 (\CC )$ and denote as usual by $V(a, \: b)$ the irreducible $G$-module whose highest weight has numerical
labels $a$, $b$ with respect to the fundamental weights $\omega_1$, $\omega_2$ determined by the choice of the torus $T$ of diagonal matrices and the
Borel subgroup $B$ of upper triangular matrices. In addition we abbreviate
\[
S^a := \mathrm{Sym}^a (\CC^3), \quad D^b := \mathrm{Sym}^b (\CC^3)^{\vee }
\] 
and introduce dual bases $e_1, \: e_2, \: e_3$ in $\CC^3$ and $x_1, \: x_2, \: x_3$ in $(\CC^3 )^{\vee }$. Recall that $V(a, \: b)$ is the kernel of the
$G$-equivariant operator
\[
\Delta\, :\, S^a \otimes D^b \to S^{a-1} \otimes D^{b-1}, \quad \Delta = \sum_{i=1}^3 \frac{\partial}{\partial e_i}\otimes \frac{\partial}{\partial
x_i}
\]
(we will always view $V(a, \: b)$ realized in this way in the following) and there is also the $G$-equivariant operator 
\[
\delta\, :\, S^{a-1} \otimes D^{b-1} \to S^{a} \otimes D^{b}, \quad \delta = \sum_{i=1}^3  e_i\otimes x_i \, .
\]
In particular,
\[
S^a \otimes D^b = \bigoplus_{i=0}^{\mathrm{min}(a, \:b)} V(a-i, \: b-i)
\]
as $G$-modules.\\
In the vast majority of cases where we apply Theorem \ref{tDoubleBundleOriginal} we will have
\begin{gather}\label{defUVW}
U:= V(e, \: 0), \quad V:= V(0, \: f), \\
W:= V(e-i_1, \: f-i_1)\oplus \dots \oplus V(e-i_m, \: f-i_m) \nonumber
\end{gather}
for some non-negative integers $e$ and $f$ and integers $0\le i_1 < i_2 < \dots < i_m \le M:= \mathrm{min}(e, \: f)$.\\
We need a fast method to compute the
$G$-equivariant map
\begin{gather}\label{formulaBilinearMap}
\psi \, :\, U  \otimes V \to W \, .
\end{gather}

\begin{remark}\xlabel{rTransposeMap}
If we know how to compute the map $\psi$ in formula \ref{formulaBilinearMap}, in the sense say, that upon choosing bases $u_1, \dots ,u_r$ in $U$,
$v_1, \dots , v_s$ in $V$, $w_1, \dots , w_t$ in $W$, we know the $t$ matrices of size $r \times s$ 
\[
M^1, \dots , M^t
\]
given by 
\[
(M^k)_{ij} := (w_k)^{\vee } ( \psi (u_i, \: v_j)) \, ,
\]
then the map
\begin{gather*}
\tilde{\psi }\, : \, W^{\vee } \otimes V \to U^{\vee}\, ,\\
\tilde{\psi }( l_W , \: v) (u) = l_W (\psi (u, \: v)) \, \: l_W \in W^{\vee}, \: v\in V, \: u\in U
\end{gather*}
induced by $\psi$ has a similar representation by $r$ matrices of size $t\times s$
\[
N^1, \dots , N^r
\]
in terms of the bases $w_1^{\vee }, \dots , w_t^{\vee }$ of $W^{\vee }$, $v_1, \dots , v_s$ of $V$, and $u_1^{\vee }, \dots , u_r^{\vee }$ of
$U^{\vee}$. In fact,
\begin{align*}
(N^{i})_{kj} &= (\tilde{\psi} (w^{\vee }_k, \: v_j)) (u_i)\\
   &= w^{\vee }_k (\psi (u_i, \: v_j)) = (M^k)_{ij}\, .
\end{align*}
The map $\tilde{\psi }$ is occasionally convenient to use instead of $\psi$.
\end{remark}

We now describe how we compute $\psi$ by writing elements of $U\otimes V$ as sums of pure tensor products of powers of linear forms. We start by proving some helpful formulas:

\begin{lemma} \xlabel{lDelta}
Let $u \in \CC^3$ and $v \in (\CC^3)^\vee$. Then
\begin{enumerate}
\item $\Delta (u^e\otimes v^f) =   e f v(u) u^{e-1}\otimes v^{f-1}$
\item $\Delta^i (u^e\otimes v^f) = \frac{e! }{(e-i)!} \frac{f!}{(f-i)!} v(u)^i  u^{e-i} \otimes v^{f-i}$
\end{enumerate}
\end{lemma}

\begin{proof} 
We can assume $v(u) \neq 0$ for otherwise $\Delta (u^e \otimes v^f) =0$. We put 
\[
u_1 := \frac{u}{v(u)}
\]
so that $v(u_1) =1$ and complete $v_1:=v$ and $u_1$ to dual bases $u_1, \: u_2, \: u_3$ in $\CC^3$ and $v_1, \: v_2, \: v_3$ in $(\CC^3 )^{\vee }$. Then
\begin{align*}
\Delta (u^e\otimes v^f) &= \left( \frac{\partial }{\partial u_1} \otimes \frac{\partial }{\partial v_1} +  \frac{\partial }{\partial u_2} \otimes
\frac{\partial }{\partial v_2} + \frac{\partial }{\partial u_3} \otimes \frac{\partial }{\partial v_3} \right) (u^e\otimes v^f) \\
   &= f\frac{\partial }{\partial u_1} \left( (v(u) u_1)^e  \right) \otimes v^{f-1}\\
   &= f e(v(u))^eu_1^{e-1}\otimes v^{f-1}\\
   &= ef  v(u) u^{e-1}\otimes v^{f-1}.
\end{align*}
This gives the first formula. Iterating it gives the second one.
\end{proof}

\begin{lemma}\xlabel{lPolynomialNature}
Let  $\pi_{e,\: f, \: i}$ be the equivariant projection
\[
\pi_{e, \: f, \: i}\, :\, S^e \otimes D^f \to V(e-i, \: f-i) \subset S^e \otimes D^f.
\]
Then one has
\[
\pi_{e, \: f,i} = \sum_{j=0}^{\min(e,f)} \mu_{i,j} \delta^j \Delta^j
\]
for certain $\mu_{i,j}\in\QQ$. 
\end{lemma}

\begin{proof} 
Set $\pi_{e,f} := \pi_{e,f,0}$ und 
look at the diagram
\xycenter{
S^e\otimes D^f \ar[r]^{\Delta^i} &S^{e-i}\otimes D^{f-i} \ar[d]^{\pi_{e-i,\: f-i}}\\
&V(e-i, \: f-i)\subset S^{e-i} \otimes D^{f-i} \ar[lu]^{\delta^i}
	}
By Schur's lemma,
\begin{gather}\label{formulaLambdai}
\pi_{e, \: f, \: i} = \lambda_i \delta^i \pi_{e-i, \: f-i} \Delta^i
\end{gather}
for some nonzero constants $\lambda_i$. On the other hand,
\[
\pi_{e, \: f} = \mathrm{id} - \sum_{i=1}^{\mathrm{min}(e, \: f)} \pi_{e, \: f, \: i} \, .
\]
Therefore, since the assertion of the Lemma holds trivially if one of $e$ or $f$ is zero, the general case follows by induction on $i$.
\end{proof}

Note that to compute the $\mu_{i,j}$ in the expression of $\pi_{e, \: f,i}$ in Lemma \ref{lPolynomialNature}, it suffices to calculate the $\lambda_i$ in
formula \ref{formulaLambdai} which can be done by the rule
\[
\frac{1}{\lambda_i} (e_1^{e-i} \otimes x_3^{f-i}) = \left( \pi_{e-i, \: f-i} \circ \Delta^i \circ \delta^{i}\right) (e_1^{e-i} \otimes x_3^{f-i})\, .
\]
Notice that applying $\delta^i \circ \Delta^i$ to a decomposable element can still yield a bihomogeneous polynomial with very many terms. A final improvement in the complexity of calculating $\psi$ is obtained by representing these bihomogeneous polynomials not by a sum of monomials but rather by their value on many points of $\CC^3 \times (\CC^3)^\vee$. Indeed such values can be calculated easily:

\begin{lemma}
Let $a,b\ge0$ be integers, $u \in \CC^3, v \in (\CC^3)^\vee,p \in (\CC^3)^\vee$ and $q \in \CC^3$. Then
\[
     \bigl( \delta^i \circ \Delta^i (u^a \otimes v^b)\bigr) (p,q) = 
      \frac{a! }{(a-i)!} \frac{b!}{(b-i)!} (\delta(p,q))^i v(u)^i  u(p)^{a-i}  v(q)^{b-i}.
\]
\end{lemma}

\begin{proof}
By Lemma \ref{lDelta} we have
\[
	 \delta^i \circ \Delta^i (u^a \otimes v^b) (p,q) = \left(\delta^i ( v(u)^i \frac{a! }{(a-i)!} \frac{b!}{(b-i)!} u^{a-i} \otimes v^{b-i}) \right)(p,q).
\]
Evaluation gives the above formula.
\end{proof}

\begin{corollary} \xlabel{cEvaluate}
Let $\psi \colon V\otimes U \to W$ be as above and assume $e\le f$. Then there
exists a homogeneous polynomial $\chi \in \QQ[x,y]$ of degree $e$, such that
\[
	\psi(u^e\otimes v^f)(p,\: q) = v(q)^{f-e} \chi\bigl(\delta(p,q)v(u),u(p)v(q)\bigr)
\]
holds for all $u \in \CC^3, v\in (\CC^3)^\vee, p \in (\CC^3)^\vee$ and $q \in \CC^3$.
\end{corollary}

\begin{proof}
We have
\[
	\psi 
	= (\pi_{e,f,i_1} + \dots + \pi_{e,f,i_m}).
\]
Using that
\[
	\pi_{e,f,i} = \sum_{j=0}^{e} \lambda_{i,j} \delta^j \Delta^j
\]
for certain $\lambda_{i,j}$ we obtain
\begin{align*}
	\psi(u^e\otimes v^f)(p,q) 
	&= \left( \sum_{\alpha = 1}^{m} \sum_{j=0}^{e} \lambda_{i_\alpha,j} \delta^j\Delta^j (u^e\otimes v^f)\right) (p,q)\\
	&=   \sum_{\alpha = 1}^{m} \sum_{j=0}^{e} \lambda_{i_\alpha,j} (\delta(p,q))^j v(u)^j \frac{e! }{(e-j)!} \frac{f!}{(f-j)!} u(p)^{e-j}  v(q)^{f-j}\\
	&=  v(q)^{f-e} \sum_{\alpha = 1}^{m} \sum_{j=0}^{e} \lambda_{i_\alpha,j} \frac{e! }{(e-j)!} \frac{f!}{(f-j)!} \bigl(\delta(p,q)v(u)\bigr)^j  \bigl(u(p)v(q)\bigr)^{e-j}\\
	&= v(q)^{f-e} \chi\bigl(\delta(p,q)v(u),u(p)v(q)\bigr) .
\end{align*}
\end{proof}
	
Now we are in a position to check the important genericity conditions of Theorem \ref{tDoubleBundleOriginal} efficiently:

\begin{proposition} \xlabel{cRank}
Let $n$ be a positive integer, $u_i \in \CC^3, v_i \in (\CC^3)^\vee, p_i \in (\CC^3)^\vee$ and $q_i \in \CC^3$ for $0 \le i \le n$. Set $x_0= \sum_{i=0}^n \xi_i u_i^e$ and consider the $n \times n$ matrix $M$ with entries
\[
	M_{j,k} = \sum_{i=0}^n \xi_i \psi(u_i^e\otimes v_j^f)(p_k,q_k).
\]
If $\rank M = \dim W$ then $\psi(x_0,V) = W$. Similarly if $y_0 = \sum_{j=0}^n \eta_j v_j^f$ and
$N$ is the $n\times n$ matrix with entries
\[
	N_{i,k} = \sum_{j=0}^n \eta_j \psi(u_i^e\otimes v_j^f)(p_k,q_k).
\]
then $\rank N = \dim W$ implies $\psi(U,y_0) = W$.
\end{proposition}

\begin{proof}
Since $\psi$ is bilinear $\psi(x_0,v_j^f) = \sum \xi_i \psi(u_i^e,v_j^f)$. Therefore
the $j$-th row  of $M$ contains the values of $\psi(x_0,v_j^f)$ at the 
points $(p_k,q_k)$ for all $k$. Therefore $\rank M \le \dim \psi(x_0,V) \le \dim W$.
If $\rank M = \dim W$ the claim follows. The second claim follows similarly.
\end{proof}

\begin{remark}
Notice the following:
\begin{enumerate}
\item  The rank condition of Proposition \ref{cRank} can also be checked over a finite field. 
\item Over a finite field all possible values of the polynomial $\chi$ can be precomputed and stored
in a table.
\item Since 
$\psi(u^e\otimes v^f)(p,q)$ can be evaluated quickly using Corollary \ref{cEvaluate} we do not have to store the $n^3$ values of this expression used in Proposition \ref{cRank}. It is enough to
store the $2n^2$ entries of $M$ and $N$. This is fortunate since $n$ must be at least $20.000$ for $d=217$ and in this case $n^3 = 8 \times 10^{12}$ values would consume about $8 GB$ of memory. 
\item Given the Polynomial $\chi$ the formula of Corollary \ref{cEvaluate} becomes so simple, that it can easily be implemented in C++. See e.g. our program {\ttfamily nxnxn} at \cite{smallDegree}.
\item Calculating the rank of a $20.000 \times 20.000$-matrix is still difficult and takes several weeks on current computers, if implemented naively. Using FFPACK \cite{FFPACK} we could distribute this work to a cluster of computers. See \cite{smallDegree} program {\ttfamily nxnxn}. For example, the case $d=210$ (the largest we computed) required 23.8 hours total run time on a machine Xeon-E5472-CPU, 8 cores.
\item The algorithm presented here is related to the one presented in \cite{BvB08-2} with the substantial improvement that the elements  of $U$ and $V$ are represented as sums of powers of linear forms and that the elements of $W$ are represented by their values. This eliminates  the need to calculate with big bihomogeneous polynomials.
\end{enumerate}
\end{remark}

\section{The Method of Covariants: Algorithms} \xlabel{sCovariantAlgorithms}

Virtually all the methods for addressing the rationality problem are based on introducing some fibration structure over a
stably rational base in the space for which one wants to prove rationality; with the Double Bundle Method, the fibres are linear, but it turns out
that fibrations with nonlinear fibres can also be useful if rationality of the generic fibre of the fibration over the function field of the base can
be proven. The \emph{Method of Covariants} (see \cite{Shep}) accomplishes this by inner linear projection of the generic fibre from a very
singular centre.
 
\begin{definition}\xlabel{dCovariants}
If $V$ and $W$ are $G$-modules for a linear algebraic group $G$, then a {\sl covariant} $\varphi$ of degree $d$ from $V$ with values in $W$ is a
$G$-equivariant polynomial map of degree $d$
\[
\varphi\, :\, V \to W\, .
\]
In other words, $\varphi$ is an element of $\mathrm{Sym}^d (V^{\vee})\otimes W$.
\end{definition}

The method of covariants phrased in a way that we find useful is contained in the following theorem.

\begin{theorem}\xlabel{tCovariants}
Let $G$ be a connected linear algebraic group the semi-simple part of which is a direct product of groups of type $\mathrm{SL}$ or $\mathrm{Sp}$. Let
$V$ and $W$ be $G$-modules, and suppose that the action of $G$ on $W$ is generically free. Let $Z$ be the ineffectivity kernel of the action of $G$ on
$W$, and assume that the action of $\bar{G} := G/Z$ is generically free on $\PP (W)$, and $Z$ acts trivially on $\PP (V)$.\\
Let 
\[
\varphi\, :\, V \to W
\]
be a (non-zero) covariant of degree $d$. Suppose the following assumptions hold:
\begin{itemize}
\item[(a)]
$\PP (W)/G$ is stably rational of level $\le \dim \PP (V) - \dim \PP (W)$.
\item[(b)]
If we view $\varphi$ as a map $\varphi\, :\, \PP (V) \dasharrow \PP (W)$ and denote by $B$ the base \emph{scheme} of $\varphi$, then there is a linear
subspace $L \subset V$ such that $\PP (L)$ is contained in $B$ together with its full infinitesimal neighbourhood of order $(d-2)$, i.e.
\[
\mathcal{I}_{B} \subset \mathcal{I}_{\PP (L)}^{d-1}\, .
\]
Denote by $\pi_L$ the projection $\pi_L \, : \, \PP (V) \dasharrow \PP (V /L)$ away from $\PP (L)$ to $\PP (V/L)$.
\item[(c)]
Consider the diagram
\xycenter{
	\PP (V) \ar@{-->}[r]^{\varphi} \ar@{-->}[d]^{\pi_{L}}& \PP (W)\\
	\PP (V/L)
	}
and assume that one can find a point $[\bar{p}] \in \PP (V/L)$ such that
\[
	\varphi |_{\mathbb{P} (L+\mathbb{C} p )} \colon \mathbb{P} (L +\mathbb{C} p)  \dasharrow  \PP (W)
\]
is dominant.
\end{itemize}
Then $\PP (V) /G$ is rational.
\end{theorem}

\begin{proof}
By assumption the group $G$ is \emph{special} (cf. \cite{Se58}), and thus $W \dasharrow W/G$ which is generically a principal $G$-bundle in the \'{e}tale
topology, is a principal bundle in the Zariski topology. Combining this with Rosenlicht's theorem on torus sections \cite{Ros}, we get that the
projection
$\PP (W)
\dasharrow
\PP (W)/G$ has a rational section
$\sigma$. Remark that property (c) implies that the generic fibre of $\pi_L$ maps dominantly to $\PP (W)$ under $\varphi$, which means that the
generic fibre of
$\varphi$ maps dominantly to $\PP  (V/L)$ under $\pi_L$, too. Note also that the map $\varphi$ becomes linear on a fibre $\PP (L +\CC g)$ because of
property (b) and that thus the generic fibre of $\varphi$ is birationally a vector bundle via $\pi_L$ over the base $\PP (V/L)$.  Thus, if we
introduce the graph
\[
	\Gamma = \overline{\{ ([q],[\bar{q}], [f]) \,|\, \pi_L ([q])=[\bar{q}], \varphi ([q])= [f]\} }
	\subset \PP (V) \times \PP\bigl( V /L \bigr) \times \PP (W)
\]
and look at the diagram
\newcommand{\rmpr}{\mathrm{pr}}
\xycenter{
	\Gamma \ar[d]_{\rmpr_{23}} \ar@{<-->}[r]^{1:1}_{\rmpr_{1}}
	& \PP (V) \ar@{-->}[r]
	& \PP (V)/\bar{G} \ar@{-->}[dd]^{\bar{\varphi}}
	\\
	\PP\bigl( V/L \bigr) \times \PP (W) \ar[d]
	\\
	\PP (W) \ar@{-->}[rr]
	& & \PP (W)/\bar{G} \ar@/^20pt/@{-->}[ll]^{\sigma}.
	}
we find that the projection $\rmpr_{23}$ is dominant and makes $\Gamma$ birationally into a vector bundle over $\PP (V/L)\times \PP (W)$. Hence $\Gamma$
is birational to a succession of vector bundles over $\PP (W)$ or has a \emph{ruled structure} over $\PP (W)$. Since $\bar{G}$ acts generically freely
on $\PP (W)$, the generic fibres of $\varphi$ and $\bar{\varphi}$ can be identified and we can pull back this ruled structure via $\sigma$ (possibly
replacing $\sigma$ by a suitable translate).  Hence  
$\PP (V) /\bar{G}$ is birational to $\PP (W) /\bar{G} \times \PP^N$ with $N = \dim \PP (V) - \dim \PP (W)$. Thus by property (a), $\PP (V)/G$ is
rational.
\end{proof}

In \cite{Shep} essentially this method is used to prove the rationality of the moduli spaces of plane curves of degrees $d\equiv 1$ (mod $9$), $d\ge
19$. In \cite{BvB08-1} it is the basis of the proof that for $d\equiv 1$ (mod $3$), $d \ge 37$, and $d\equiv 2$ (mod $3$), $d\ge 65$, these moduli
spaces are rational. We improve these bounds here substantially and now recall the results from \cite{BvB08-1} which we use in our algorithms.\\
In that paper we used Theorem \ref{tCovariants} with the following data: $G$ is $\mathrm{SL}_3 (\CC )$ throughout.
\begin{itemize}
\item
For $d= 3n+1$, $n\in\mathbb{N}$, and $V= V(0, \: d)= \mathrm{Sym}^d (\CC^3)^{\vee }$, we take $W = V(0,\: 4)$ and produce covariants
\[
S_d\, :\, V(0,\: d)  \to V(0, \: 4)
\]
of degree $4$. We show that property (b) of Theorem \ref{tCovariants} holds for the space
\[
L_S = x_1^{2n+3} \cdot \CC[x_1,x_2,x_3]_{n-2} \subset V(0, \: d)\, .
\]
Moreover, $\PP (V(0,\: 4))/G$ is stably rational of level $8$. So for particular values of $d$, it suffices to check property (c) by explicit
computation. We give the details how this is done below.
\item
For $d= 3n+2$, $n\in\mathbb{N}$, and $V= V(0, \: d)= \mathrm{Sym}^d (\CC^3)^{\vee }$, we take $W = V(0,\: 8)$ and produce covariants
\[
T_d\, :\, V(0,\: d)  \to V(0, \: 8)
\]
again of degree $4$. In this case, property (b) of Theorem \ref{tCovariants} can be shown to be true for the subspace
\[
L_T = x_1^{2n+5} \cdot \CC[x_1,x_2,x_3]_{n-3} \subset V(0, \: d)\, .
\]
$\PP (V(0,\: 8))/G$ is stably rational of level $8$, too, hence again everything comes down to checking 
property (c) of Theorem \ref{tCovariants}. 
\end{itemize}

We recall from \cite{BvB08-1} how some elements of $L_S$ (resp. $L_T$) can be written as sums of powers of linear forms which is very useful for
evaluating $S_d$ resp. $T_d$ easily. Let $K$ be a positive integer.
\begin{definition}
Let $\mathbf{b} = (b_1,\dots , b_K)\in\CC^K$ be given. Then 
we denote by 
\begin{gather}
p_i^{\mathbf{b}}(c) := \prod_{\stackrel{j\neq i}{1 \le j\le K}} \frac{c-b_j}{b_i-b_j}
\end{gather}
for $i=1,\dots , K$ 
the interpolation polynomials of degree $K-1$ w.r.t. $\mathbf{b}$ in the one variable $c$. 
\end{definition}

Then we have the following easy Lemma (see \cite{BvB08-1}, Lemma 5.2, for a proof)

\begin{lemma} \xlabel{lConstruction}
Let $\mathbf{b} = (b_1,\dots , b_K)\in\CC^K$, $b_i\neq b_j$ 
for $i\not=j$, and set 
$x=x_1$, $y=\lambda x_2 +\mu x_3$, $(\lambda,\: \mu )\neq (0,\: 0)$. Suppose $d > K$ and 
put $l_i:= b_i x +y$. Then for each $c\in\CC$ with $c\neq b_i$, $\forall i$,
\begin{gather}
f(c) = p_1^{\mathbf{b}}(c)l_1^d +\dots + p_K^{\mathbf{b}}(c) l_K^d - (cx + y)^d
\end{gather}
is nonzero and 
divisible by $x^K$.
\end{lemma}

So for $K=2n+3$ we obtain elements in $f(c) \in L_S$ and for $K=2n+5$ elements $f(c)\in L_T$. We now check property (c) of Theorem \ref{tCovariants}
computationally in the following way. We choose a fixed $g \in V(0, \: d)$ which we write as a sum of powers of linear forms
\[
g = m_1^d +\dots + m_{\mathrm{const}}^d
\]
where $\mathrm{const}$ is a positive integer. We choose a random vector $\mathbf{b}$, random $\lambda$ and $\mu$, and a random $c$, and use formula
(30) from \cite{BvB08-1} which reads
\begin{align*}
S_d (f(c) + g) &= 
                        	 \sum_{i,j,k,p} p^{\mathbf{b}}_i(c)I(l_i,\: m_j, \: m_k, \: m_p)^n  l_i m_jm_km_p 
			\nonumber \\
	               	&
			\quad\quad 
			+S_d(-(cx+y)^d +g)
\end{align*}
to evaluate $S_d$. Here $I$ is a function on quadruples of linear forms to $\CC$: if in coordinates 
\begin{gather*}
L_{\alpha } = \alpha_1 x_1 + \alpha_2 x_2 + \alpha_3 x_3
\end{gather*}
and $L_{\beta }$, $L_{\gamma }$, $L_{\delta}$ are linear forms defined analogously, and if we moreover abbreviate
\[
(\alpha \: \beta \: \gamma ) := \det \left( \begin{array}{ccc}
\alpha_1 & \alpha_2 & \alpha_3\\
\beta_1 &\beta_2 &\beta_3\\
\gamma_1 &\gamma_2 &\gamma_3
\end{array}\right) \quad \mathrm{etc.,}
\]
as in the symbolic method of Aronhold and Clebsch \cite{G-Y}, then
\[
I( L_{\alpha}, \: L_{\beta}, \: L_{\gamma}, \: L_{\delta}):= (\alpha \beta \gamma ) (\alpha \beta \delta )( \alpha \gamma\delta )(\beta\gamma\delta )\,
.
\]
For $T_d$ we have by an entirely analogous computation 
\begin{align}
T_d (f(c) + g) &= 
                        	 \sum_{i,j,k,p} p^{\mathbf{b}}_i(c)I(l_i,\: m_j, \: m_k, \: m_p)^n  l_i^2 m_j^2m_k^2m_p^2 
			\nonumber \\
	               	&
			\quad\quad 
			+T_d(-(cx+y)^d +g)
\end{align}
So we can evaluate $T_d$ similarly. Thus for each particular value of $d$ we can produce points in $\PP (V(0, \:4))$, for $d=3n+1$, or $\PP (V(0, \:
8))$, for $d=3n+2$, which are in the image of the restriction of $S_d$ to a fibre of $\pi_{L_S}$ resp. in the image of the restriction of $T_d$ to a
fibre of $\pi_{L_T}$. We then check that these span $\PP (V(0, \: 4))$ resp. $\PP (V(0, \: 8))$ to check condition (c) of Theorem \ref{tCovariants}.

\section{Applications to Moduli of Plane Curves} \xlabel{sApplications}

The results on the moduli spaces of plane curves $C(d)$ of degree $d$ that we obtain are described below. We organize them according to the method
employed.

\

\bf{Double Bundle Method.}\;\mdseries As we mentioned above, Katsylo obtained in \cite{Kat89} the rationality of $C(d)$, $d\equiv 0$ (mod $3$) and $d
\ge 210$. Using the computational scheme of Section \ref{sDoubleBundleAlgorithms} and our program {\ttfamily nxnxn} at \cite{smallDegree}, we obtain the rationality of all $C(d)$ with $d\equiv 0$ (mod $3$) and $d \ge 30$ except $d=48,54,69$.
Moreover, we obtain rationality for $d=10$ and $d=21$ (the latter was known before, since by the results of \cite{Shep}, $C(d)$ is rational
for $d\equiv 1$ (mod $4$)).  A table of $U$, $V$ and $W$ used in each case can be found at \cite{smallDegree}, {\ttfamily UVW.html}. We found these combinatorially using our program {\ttfamily alldimensions2.m2} at\cite{smallDegree}.  

For $d=69$ the result is known by \cite{Shep} since $69 \equiv 1$ (mod $4$). For the cases $d=27$ and $d=54$ we need more special $U$, $V$, $W$ and use the methods from our article \cite{BvB08-2}.

\

\
\bf{The case $d=27$.}\;\mdseries We establish the rationality of $C(27)$ as follows: there is a bilinear, $\mathrm{SL}_3 (\CC )$-equivariant map
\[
\psi\, :\, V(0,\: 27) \times (V(11, \: 2)\oplus V(15, \: 0)) \to V(2, \: 14)
\]
and
\begin{gather*}
\dim V(0,\: 27) = 406, \;  \dim V(11, \: 2)= 270,\\
\dim V(15,\: 0)= 136, \; \dim V(2, \: 14) = 405\, .
\end{gather*}
We compute $\psi$ by the method of \cite{BvB08-2} and find that 
$\psi = \omega^2\beta^{11} \oplus \beta^{13}$ 
in the notation of that article. For a random $x_0 \in V(0,\: 27)$, the kernel of $\psi (x_0, \cdot )$ turns out to be
one-dimensional, generated by $y_0$ say, and $\psi (\cdot , \: y_0)$ has likewise one-dimensional kernel generated by $x_0$ (See\cite{smallDegree}, {\ttfamily degree27.m2} for a Macaulay script doing this calculation). It follows that
the map induced
by $\psi$
\[
\PP (V(0, \: 27)) \dasharrow \PP (V(11, \: 2)\oplus V(15, \: 0))
\]
is birational, and it is sufficient to prove rationality of $\PP (V(11, \: 2)\oplus V(15, \: 0))/\mathrm{SL}_3 (\CC )$. But $\PP (V(11, \: 2)\oplus
V(15, \: 0))$ is birationally a vector bundle over $\PP (V(15, \: 0))$, and $\PP (V(15, \: 0))/\mathrm{SL}_3 (\CC )$ is stably rational of level $19$,
so $\PP (V(11, \: 2)\oplus V(15, \: 0))/\mathrm{SL}_3 (\CC )$ is rational by the no-name lemma \ref{lNoNameLemma}.

\begin{table} 
\begin{tabular}{|c||l|} \hline
Degree $d$ of curves  & Result and method of proof/reference \\ \hline \hline
1 & rational (trivial) \\ \hline
2 & rational (trivial) \\ \hline
3 & rational (moduli space affine $j$-line) \\ \hline
4 & rational, \cite{Kat92/2}, \cite{Kat96}  \\ \hline
5 & rational, two-form trick \cite{Shep} \\ \hline
6 & rationality unknown \\ \hline
7 & rationality unknown \\ \hline
8 & rationality unknown \\ \hline
9 & rational, two-form trick \cite{Shep} \\ \hline
10 & rational, double bundle method, this article \\ \hline
11 & rationality unknown \\ \hline
12 & rationality unknown \\ \hline
13 & rational, two-form trick \cite{Shep}  \\ \hline
14 & rationality unknown \\ \hline
15 & rationality unknown  \\ \hline
16 & rationality unknown  \\ \hline
17 & rational, two-form trick \cite{Shep}  \\ \hline
18 & rationality unknown  \\ \hline
19 & Covariants, \cite{Shep} and this article\\ \hline
20 & rationality unknown \\ \hline
21 & rational, two-form trick \cite{Shep}  \\ \hline
22 & Covariants, this article\\ \hline
23 & rationality unknown \\ \hline
24 & rationality unknown \\ \hline
25 & rational, two-form trick \cite{Shep}  \\ \hline
26 & rationality unknown\\ \hline
27 & rational, this article (method cf. above) \\ \hline
28 & Covariants, \cite{Shep} and this article \\ \hline
29 & rational, two-form trick \cite{Shep}  \\ \hline
30 & double bundle method, this article\\ \hline
31 & Covariants, this article\\ \hline
32 & rationality unknown\\ \hline
$\ge 33$ (excl. 48) & rational, this article, \cite{BvB08-1}, \cite{Kat89} \\ \hline
\end{tabular}
\caption{Table of known rationality results for $C(d)$} \label{tKnownResults}
\end{table}

\bf{The case $d=54$.}\;\mdseries We establish the rationality of $C(54)$ as follows: there is a bilinear, $\mathrm{SL}_3 (\CC )$-equivariant map
\[
\psi \colon V(0,54) \times \bigl(
V(11,8) \oplus V(6,3) \oplus V(5,2) \oplus V(3,0) 
\bigr) \to V(0,51)
\]
with
\begin{gather*}
\dim V(0,\: 54) = 1540, \;  \dim V(11, \: 8)= 1134, \dim V(6,\: 3)= 154, \\
\; \dim V(5, \: 2) = 81, \dim V(3,\: 0) = 10, \dim V(0,51) = 1378
\end{gather*}
Since $1134+154+81+10 = 1379 = 1378+1$ and $1540-1379 > 19$ we only need to check the genericity condition of 
Theorem \ref{tDoubleBundleOriginal} to prove rationality. For this
we compute $\psi$ by the method of \cite{BvB08-2} and find that 
$\psi = \beta^{11} \oplus \beta^{6} \oplus \beta^{5} \oplus \beta^{3}$ 
in the notation of that article. 

For a random $x_0 \in V(0,\: 54)$, the kernel of $\psi (x_0, \cdot )$ turns out to be
one-dimensional, generated by $y_0$ say, and $\psi (\cdot , \: y_0)$ has full rank  $1378$
and therefore $\psi(V(0,54),\: y_0) = V(0,51)$ as required. See \cite{smallDegree}, {\ttfamily degree54.m2} for a Macaulay script doing this calculation.

\
 
\bf{Method of Covariants.}\;\mdseries According to \cite{BvB08-1}, $C(d)$ is rational for $d\equiv 1$ (mod $3$), $d\ge 37$, and $d\equiv 2$ (mod
$3$), $d\ge 65$ (for $d\equiv 1$ (mod $9$), $d \ge 19$, rationality was proven before in \cite{Shep}). By the method of Section \ref{sCovariantAlgorithms}, we improve this
and obtain that $C(d)$ is rational for $d\equiv 1$ (mod $3$), $d\ge 19$, which uses the covariants $S_d$ of Section \ref{sCovariantAlgorithms}, and rational for  $d\equiv 2$
(mod $3$), $d\ge 35$, which uses the family of covariants $T_d$ of Section \ref{sCovariantAlgorithms}. See \cite{smallDegree}, {\ttfamily interpolation.m2} for a Macaulay Script doing this calculation.

\

Combining what was said above with the known rationality results for $C(d)$ for small values of $d$, we can summarize the current knowledge in Table \ref{tKnownResults}. Thus we obtain our main theorem:

\begin{theorem}\xlabel{tComprehensive}
The moduli space $C(d)$ of plane curves of degree $d$ is rational except possibly for one of the values in the following list: \[ d= 6, \: 7, \: 8, \:
11, \: 12,\:  14,\: 15,\: 16,\: 18,\: 20,\: 23,\: 24,\: 26,\: 32,\: 48 \, . \]
\end{theorem}

\end{document}